\documentclass[12pt]{article}

\usepackage[english]{babel}
\usepackage{graphicx}
\usepackage{latexsym}
\usepackage{amsfonts}
\usepackage{dsfont}
\usepackage{amscd}
\usepackage{amssymb}
\usepackage{amsmath}
\usepackage{hyperref}
\usepackage{mathrsfs}
\usepackage{amsthm}
\usepackage{tikz}
\usepackage{enumitem}

\allowdisplaybreaks

\setlist[enumerate,1]{label=\textup{(\roman*)}}

\newlength{\bibitemsep}
\newlength{\bibparskip}
\let\oldthebibliography\thebibliography
\renewcommand\thebibliography[1]{%
  \oldthebibliography{#1}%
  \setlength{\parskip}{\bibitemsep}%
  \setlength{\itemsep}{\bibparskip}%
}

\newcommand{\NN}{\mathbb{N}}
\newcommand{\CC}{\mathbb{C}}

\newcommand{\Q}{\mathcal{Q}}
\newcommand{\A}{\mathfrak{a}}

\newcommand{\ev}{\text{ev}}

\newcommand{\dom}{\text{dom}}
\newcommand{\ran}{\text{ran}}
\newcommand{\ep}{\varepsilon}

\newcommand{\id}{\text{Id}}

\newcommand{\iso}{\text{Iso}}

\newcommand{\I}{\mathcal{I}}
\newcommand{\W}{\mathcal{W}}
\newcommand{\J}{\mathcal{J}}
\newcommand{\s}{\mathcal{S}}

\newcommand{\B}{\mathcal{B}}

\newcommand{\X}{\mathsf{X}}

\newcommand{\OX}{\mathcal{O}_\X}
\newcommand{\g}{\mathcal{G}}
\renewcommand{\l}{\mathcal{L}}

\newcommand{\Ef}{\widehat E_0}

\newcommand{\Eu}{\widehat E_\infty}
\newcommand{\Et}{\widehat E_{\text{tight}}}
\newcommand{\Ct}{C^*_{\text{tight}}}

\newcommand{\gt}{\mathcal{G}_{\text{tight}}}
\newcommand{\fs}{\subseteq_{\text{fin}}}
\newcommand{\vp}{\varphi}

\newcommand{\f}{\mathcal{F}}
\newcommand{\go}{\g^{(0)}}

\setlist[enumerate]{itemsep=-1pt}
\setlist[itemize]{itemsep=-1pt}
\theoremstyle{definition}
\newtheorem{theo}{Theorem}[section]
\newtheorem{rmk}[theo]{Remark}

\newtheorem{lem}[theo]{Lemma}

\newtheorem{prop}[theo]{Proposition}
\date{}
\begin{document}

\title{A new uniqueness theorem for the tight C*-algebra of an inverse semigroup}
\author{Charles Starling\thanks{Partially supported by an NSERC Discovery Grant and an internal Carleton University research grant. \texttt{cstar@math.carleton.ca}} }
\maketitle

\begin{abstract}
We prove a new uniqueness theorem for the tight C*-algebras of an inverse semigroup by generalising the uniqueness theorem given for \'etale groupoid C*-algebras by Brown, Nagy, Reznikoff, Sims, and Williams. We use this to show that in the nuclear and Hausdorff case, a *-homomorphism from the boundary quotient C*-algebra of a right LCM monoid is injective if and only if it is injective on the subalgebra generated by the core submonoid. We also use our result to clarify the identity of the tight C*-algebra of an inverse semigroup we previously associated to a subshift and erroneously identified as the Carlsen-Matsumoto algebra.
\end{abstract}

MSC Classes: 46L05, 20M18, 18B40


\section{Introduction}

An inverse semigroup is a semigroup $S$ where every element $s$ has a unique inverse $s^*$ in the sense that $s^*ss^* = s^*$ and $ss^*s = s$, and a representation of $S$ in a C*-algebra is a map $\rho:S\to A$ that is multiplicative and sends inverses to adjoints. Following Paterson \cite{Pa02}, Exel \cite{Ex08} defined a C*-algebra $\Ct(S)$ generated by a copy of $S$ which is universal for what he called {\em tight representations}, which enforce a kind of nondegeneracy condition. This construction has been shown to unify many classes of C*-algebras: graph, $k$-graph, and Cuntz-Krieger algebras \cite{Ex08}, Katsura algebras, Nekrashevych algebras \cite{EP17} , AF algebras \cite[Remark~6.5]{LS17}, \cite[Theorem~5.1]{St16}, tiling algebras \cite{EGS12}, Cuntz-Li algebras of certain semigroups \cite{StLCM}, and many others have been shown to be isomorphic to $\Ct(S)$ for some inverse semigroup $S$ defined from the underlying combinatorial data. If one identifies a C*-algebra $A$ with elements satisfying the same relations as those in $\Ct(S)$, then the universal property implies the existence of a $*$-homomorphism $\pi_\rho: \Ct(S)\to A$. A natural problem which then arises is determining when $\pi_\rho$ is injective; this is the main problem we consider in this paper.  

In some contexts, it is possible that injectivity of $\pi_\rho$ is equivalent to the injectivity of $\pi_\rho$ on some subset of $S$. A famous example is the {\em Cuntz-Krieger uniqueness theorem} for graph C*-algebras, which as noted above are algebras which can be realized as $\Ct(S)$ for some $S$. To illustrate our perspective, we translate the Cuntz-Krieger uniqueness theorem to the language of tight representations. Given a row-finite directed graph $E$ with no sources with edge set $E^1$ and vertex set $E^0$, one says that a set of partial isometries $\{s_e\}_{e\in E^1}$ and mutually orthogonal projections $\{p_v\}_{v\in E^0}$ is a {\em Cuntz-Krieger $E$-family} if they satisfy 
\begin{enumerate}
	\item[(CK1)] $s_es_e^* = p_{s(e)}$ for all $e\in E^1$
	\item[(CK2)] $\sum_{e\in r^{-1}(v)} s_es_e^* = p_v$ for all $v\in E^0$.  
\end{enumerate}
These are the {\em Cuntz-Krieger relations}.  One can associate to $E$ an inverse semigroup $S_E$ generated by elements $\{S_e\}_{e\in E^1}$ and mutually orthogonal idempotents  $\{P_v\}_{v\in E^0}$ subject to (CK1) but not (CK2) (because there is no ``sum'' in a general inverse semigroup). Then $\rho$ is a representation of $S_E$ if and only if the images of the generators satisfy (CK1), and will be tight if and only if the images of the generators satisfy (CK2). Then saying that $\Ct(S_E)$ is universal for tight representations is the same as saying it is universal for generators satisfying the Cuntz-Krieger relations, and so it is isomorphic to $C^*(E)$, the {\em graph C*-algebra} for $E$. In this language, the Cuntz-Krieger uniqueness theorem says that if every cycle in $E$ has an entry, then given a tight representation $\rho$ of $S_E$ on $A$, the induced $*$-homomorphism $\pi_\rho:C^*(E)\to A$ is injective if and only if $\rho(P_v)\neq 0$ for all $v\in E^0$. This is a powerful tool for analyzing these algebras, as the potentially complicated task of proving injectivity of a $*$-homomorphism is reduced to checking injectivity on a subset of the generators.

The original proof of the Cuntz-Krieger uniqueness theorem \cite{KPR98} uses the realization of $C^*(E)$ as the reduced C*-algebra $C^*_r(\g_E)$ of a certain groupoid $\g_E$ associated to $E$. The condition that every cycle has an entry was seen to be equivalent to $\g_E$ being {\em effective}, which means that the interior $\iso(\g_E)^\circ$ of the isotropy subgroupoid of $\g_E$ is equal to the unit space $\g_E^{(0)}$. A key step in their proof is showing that a $*$-homomorphism from the reduced C*-algebra of a locally compact Hausdorff effective \'etale 0-dimensional groupoid $\g$ is injective if and only if it is injective on $C_0(\go)$. This step together with the final theorem suggest that perhaps the assumption that $\g_E$ is effective could be dropped, and that injectivity of a $*$-homomorphism from $C^*_r(\g)$ is equivalent to injectivity on the subalgebra $C_r^*(\iso(\g)^\circ)$. This is the content of the seminal result of Brown, Nagy, Reznikoff, Sims, and Williams \cite{BNRSW16}, who were even able to prove this without the assumption that $\g$ is 0-dimensional. This has since been applied in many contexts \cite{CN20, LM17, LMS19, CEP18}, and generalized in \cite{Arm22} to twisted groupoid C*-algebras.

Returning to $\Ct(S)$, one has that it can always be realized as the C*-algebra of an \'etale groupoid $\gt(S)$, and so \cite[Theorem~3.1]{BNRSW16} applies. In \cite{LM17}, Milan and LaLonde applied it to $\Ct(S)$ (and indeed other inverse semigroup algebras) and showed that as long as $S$ is {\em 0-disjunctive} and $\Ct(S)$ is nuclear, then any $*$-homomorphism from $\Ct(S)$ is injective if and only if it is injective on the subalgebra generated by the {\em centralizer} $Z(S)$ of $S$ (see Lemma~\ref{lem:0dis} and the discussion before it for the definitions).

 We revisit the line of work started in \cite{LM17} because there are some classes of inverse semigroups of interest (for example, the ones we consider in Section~\ref{sec:LCM}) which are not necessarily 0-disjunctive. We identify an inverse subsemigroup $S^\iso\subseteq S$ which takes the role of $Z(S)$ in \cite{LM17}, and show in Theorem~\ref{thm:isguniqueness} that in the nuclear case, a $*$-homomorphism from $\Ct(S)$ is injective if and only if it is injective on the subalgebra generated by $S^\iso$. As in \cite{LM17}, we must also put a condition on $S$ which guarantees the tight groupoid is Hausdorff, which we call \eqref{eq:H}. To prove our theorem, we first need a slight generalization of \cite[Theorem~3.1]{BNRSW16}. Briefly, our Theorem~\ref{th:uniqueness} says that one can check injectivity on certain open subgroupoids $\f\subseteq \iso(\g)^\circ$ as long as there are enough units with full isotropy in $\f$. We also give a condition on $S$ which guarantees that $\iso(\gt(S))^\circ$ is closed, which by \cite[Proposition~4.1]{BNRSW16} guarantees the existence of a faithful conditional expectation onto $\Ct(S^\iso)$ (in the nuclear case). We give a formula of this map on the generators, Proposition~\ref{prop:CondExp}.
 
 We have two main applications. The first is that of Cuntz-Li algebras associated to LCM monoids. In \cite{Li12}, Li generalized the work of Nica \cite{Ni92} to define a universal C*-algebra $C^*(P)$ and a boundary quotient $\Q(P)$ to a left-cancellative semigroup $P$. A class of semigroups of particular interest are the right LCM semigroups, whose C*-algebras have been considered in many works \cite{ABLS19, BOS18, BLS16, BLS18, BS16, LL20, LL21, LiB19, NS19, Stam15, Stam17, StLCM}. In the monoid case, $P$ is right LCM if $pP\cap qP$ is either empty or equal to $rP$ for some $r\in P$ (here, $r$ is a right least common multiple of $p$ and $q$). Then the {\em core submonoid} is the set $P_c$ of elements $p$ such that $pP\cap qP$ is nonempty for all $q\in P$. This submonoid was defined in \cite{CL07, StLCM} and features prominently in many of the works mentioned above. We prove in Theorem~\ref{th:LCMmain} that in the nuclear and Hausdorff case that a $*$-homomorphism from $\Q(P)$ is injective if and only if it is injective on $\Q_c(P)$, the subalgebra generated by $P_c$. 
 
 Our second main application is to subshift algebras. If $\A$ is a finite alphabet and $\A^\NN$ is the Cantor space of one-sided sequences in $\A$, a {\em subshift} is a closed subspace $\X\subseteq \A^\NN$ which is invariant under the map $\sigma$ which shifts a sequence one position to the left. There are many C*-algebras associated to such objects, defined by Matsumoto and later Carlsen \cite{Ca08, CM04,CS07, Ma97, Ma99}. Two of these C*-algebras feature in this work: $\OX$, a certain universal C*-algebra defined in \cite{Ca08}, and the C*-algebra generated by translation operators $S_a:\ell^2(\X)\to \ell^2(\X)$ given by $S_a(\delta_x) = \delta_{ax}$ if $ax\in \X$ and 0 otherwise, for $a\in \A$. In \cite{ShiftISG}, we associated an inverse semigroup $\s_\X$ to a subshift and showed (erroneously) that $\Ct(\s_\X)$ is isomorphic to $\OX$. In a submitted corrigendum, we show this isomorphism holds when $\X$ is assumed to satisfy Matsumoto's {\em condition (I)}. In this paper, we can use our Theorem~\ref{thm:isguniqueness} to show that whether one assumes condition (I) or not, we can conclude that $\Ct(\s_\X)$ is isomorphic to $C^*(S_a: a\in \A)$, see Theorem~\ref{th:subshiftisg}. 
 
 This paper is organized in the following manner. In Section~\ref{sec:GpdUniqueness} we give background on \'etale groupoids and prove our generalization (Theorem~\ref{th:uniqueness}) of the uniqueness theorem proved by Brown, Nagy, Reznikoff, Sims, and Williams. In Section~\ref{sec:ISG} we give background on inverse semigroups and define their tight groupoids and corresponding tight C*-algebras, then use Theorem~\ref{th:uniqueness} to prove our uniqueness theorem for tight C*-algebras of inverse semigroups (Theorem~\ref{thm:isguniqueness}). We also give conditions on $S$ which imply existence of a conditional expectation onto $\Ct(S^\iso)$ via \cite[Proposition~4.1]{BNRSW16}. In Section~\ref{sec:LCM} we apply our results to the boundary quotient C*-algebras of right LCM monoids, and in Section~\ref{sec:subshifts} we apply them to subshift C*-algebras.

\section{A slightly improved uniqueness theorem for groupoid C*-algebras}\label{sec:GpdUniqueness}

We will use the following general notation. If $X$ is a set and $U\subseteq X$, let Id$_U$ denote the map from $U$ to $U$ which fixes every point, and let $1_{U}$ denote the characteristic function on $U$, i.e. $1_U: X\to \CC$ defined by $1_U(x) = 1$ if $x\in U$ and $1_U(x) = 0$ if $x\notin U$. If $F$ is a finite subset of $X$, we write $F\fs X$. If $X$ is a topological space and $U\subseteq X$, we write $U^\circ$ for the interior of $U$. 

A {\em groupoid} is a set $\g$ equipped with a partially defined product such that every element has a ``local'' inverse. Specifically, there exists a set $\g^{(2)}\subseteq \g\times \g$, called the set of {\em composable pairs} on which there is a product $\g^{(2)}\ni (\gamma,\delta) \mapsto \gamma\delta\in \g$. Furthermore, this product is associative (wherever defined) and for each $\gamma\in \g$ there exists a (unique) element $\gamma^{-1}\in \g$ such that $\gamma^{-1}\gamma\eta = \eta$ and $\delta\gamma\gamma^{-1} = \delta$ whenever $(\gamma, \eta), (\delta, \gamma)\in \g^{(2)}$. The set of {\em units} is $\go = \{\gamma\gamma^{-1}: \gamma\in \g\} = \{\gamma^{-1}\gamma: \gamma\in \g\}$, and the maps $r(\gamma) = \gamma\gamma^{-1}$ and $d(\gamma) = \gamma^{-1}\gamma$ from $\g$ to $\go$ are called the {\em range} and {\em source} maps respectively.\footnote{In most sources the source map is denoted $s$, but this variable is used so often in inverse semigroup literature we reserve it for elements of an inverse semigroup.} A subset $\f\subseteq \g$ is called a {\em subgroupoid} if it is closed under the product and inverse.

For $u, v\in \go$ and a subset $\l\subseteq \g$ we use the notation
\[
\l_u = d^{-1}(u)\cap \l,\hspace{1cm}\l^u = r^{-1}(u)\cap \l,\hspace{1cm}\l_u^v = \l^v\cap \l_u.
\]
The set $\g_u^u$ is a group with unit $u$ and is called the {\em isotropy group} at $u$, and the set $\iso(\g):=\cup_{u\in \go}\g_u^u$ is called the {\em isotropy group bundle} of $\g$. It is a subgroupoid of $\g$ containing the units. 

A groupoid $\g$ with a topology is called a {\em topological groupoid} if the inverse and product maps are continuous (where $\g^{(2)}$ is given the product topology). A topological groupoid is called {\em \'etale} if $\go$ is Hausdorff and the map $r:\g\to \go$ is a local homeomorphism, i.e. for each $\gamma\in \g$ there exists an open set $B$ around $\gamma$ such that $r(B)$ is open in $\go$ and $r: B\to r(B)$ is a homeomorphism. One sees that this implies $d$ is a local homeomorphism as well. Any open set $B$ on which $r$ is a homeomorphism between open sets is called a {\em bisection}. The unit space of an \'etale groupoid is always open, and will be closed if and only if $\g$ is Hausdorff.

All the groupoids considered in this work will be locally compact, Hausdorff, and \'etale. 

We now describe the C*-algebras associated to such groupoids, as defined in \cite{R80} and described in \cite[Sections~3.2--3.3]{Sim20}. If $\g$ is a locally compact Hausdorff groupoid, let $C_c(\g)$ denote the compactly supported continuous functions from $\g$ to $\CC$. This is a complex $*$-algebra when given pointwise sum and scalar multiplication, and the operations
\[
fg(\gamma) = \sum_{\gamma_1\gamma_2 = \gamma}f(\gamma_1)g(\gamma_2), \hspace{1cm}f^*(\gamma) = \overline{f(\gamma^{-1})}.
\] 
For $u\in \go$, let $\pi_u: C_c(\g)\to \B(\ell^2(\g_u))$ denote the representation of $C_c(\g)$ given by
\[
\pi_u(f)\delta_\gamma = \sum_{\alpha\in \g_{r(\gamma)}}f(\alpha)\delta_{\alpha\gamma}, \hspace{1cm}f\in C_c(\g)
\]
Then the {\em reduced C*-algebra of $\g$}, denoted $C^*_r(\g)$, is the completion of the image of $C_c(\g)$ under $\bigoplus_{u\in\go}\pi_u$. The C*-norm on this C*-algebra, called the {\em reduced norm}, is given on elements of $C_c(\g)$ by $\|f\|_r = \sup_{u\in\go}\{\|\pi_u(f)\|\}$. There is another C*-algebra associated to $\g$, called the {\em full C*-algebra} of $\g$ which has $C_c(\g)$ as a dense *-subalgebra and which is universal for representations of $C_c(\g)$. The C*-norm on $C^*(\g)$, called the {\em full norm}, is given on elements of $C_c(\g)$ by $\|f\| = \sup\{\|\pi(f)\|: \pi\text{ is a representation of }C_c(\g)\}$. Evidently $\bigoplus_{u\in\go}\pi_u$ is a representation of $C_c(\g)$, so the universal property of $C^*(\g)$ gives a surjective $*$-homomorphism $\Lambda: C^*(\g)\to C^*_r(\g)$ called the {\em left regular representation}. If $\Lambda$ is an isomorphism, we say $\g$ satisfies {\em weak containment}; this holds in particular if $\g$ is {\em amenable} \cite{AR00} but can also hold in some cases when $\g$ is not amenable \cite{AF18, Wi15}.

A subset $X\subseteq \go$ is called {\em invariant} if $d(r^{-1}(X)) = X$. If $X$ is closed and invariant, then $\g^X_X := \bigcup_{x,y\in X}\g_x^y$ is a groupoid with unit space $X$, and there exists a surjective $*$-homomorphism $q_X: C^*_r(\g)\to C^*_r(\g_X^X)$ given on $C_c(\g)$ by function restriction. When $X = \{u\}$ is a singleton and is invariant, we denote this quotient map $q_u: C^*_r(\g)\to C^*_r(\g_u^u)$. There is also a corresponding quotient map between the full C*-algebras. If $\{u\}$ is invariant and we have a subalgebra $A\subseteq C^*_r(\g)$  with a state $\vp$, we will say the state {\em factors through} $C^*_r(\g_u^u)$ if there is a state $\psi$ on $C^*_r(\g_u^u)$ such that $\vp = \psi\circ q_u$. 

Let $\g$ be a locally compact Hausdorff \'etale groupoid. If $\f\subseteq \g$ is an open subgroupoid, then there are injective $*$-homomorphisms $\iota_{\f, r}: C^*_r(\f)\to C^*_r(\g)$ and $\iota_\f: C^*(\f)\to C^*(\g)$ such that for all $f\in C_c(\f)$, $\iota_{\f, r}(f)$ and $\iota_\f(f)$ are both the function obtained by extending $f$ to $\g$ by defining it to be 0 outside of $\f$.  In this section we will prove the following generalization of \cite[Theorem~3.1]{BNRSW16}.

\begin{theo}
\label{th:uniqueness}
Let $\g$ be a locally compact Hausdorff \'etale groupoid, and let $\f\subseteq \iso(\g)^\circ$ be an open subgroupoid containing the unit space. 
\end{theo}
\begin{enumerate}
	\item \label{it1:uniqueness}Suppose that $u\in \go$ satisfies  $\f^u_u = \g_u^u$. If $\vp$ is a state on $\iota_{\f, r}(C^*_r(\f))$ that factors through $C^*_r(\g^u_u)$, then $\vp$ has a unique extension $\widetilde\vp$ to $C^*_r(\g)$. If $\vp$ is a state on $\iota_{\f}(C^*(\f))$ that factors through $C^*(\g^u_u)$, then $\vp$ has a unique extension $\widetilde\vp$ to $C^*(\g)$.   
	\item \label{it2:uniqueness} Suppose that $X_\f:= \{u\in \go : \f^u_u = \g^u_u \}$ is dense in $\go$. Then for any $*$-homomorphism $\phi: C_r^*(\g)\to B$, $\phi$ is injective if and only if its restriction to $\iota_{\f, r}(C^*_r(\f))$ is injective. 
\end{enumerate}

Before proving Theorem~\ref{th:uniqueness}, we wish to give an idea of why we want such a generalization. In the next section we will consider a groupoid of the form $\gt(S)$ for an inverse semigroup $S$, and construct an inverse subsemigroup $S^\iso\subseteq S$ whose tight groupoid $\gt(S^\iso)$ sits inside $\iso(\gt(S))^\circ$ as an open subgroupoid containing the units. Under some conditions they are equal, but not in general. Nevertheless, Theorem~\ref{th:uniqueness} will allow us to show that a $*$-homomorphism on $C^*_r(\gt(S))$ is injective if and only if it is injective on the subalgebra generated by the generators corresponding to $S^\iso$. 

Another point which must be acknowledged before getting to the proof is that the arguments from \cite{BNRSW16} go through in our case, essentially unchanged. So in what follows we briefly run through these arguments, pointing out along the way what modifications need to be made. 

We first prove a lemma which is almost identical to \cite[Lemma~3.5]{BNRSW16}, but for our proof of Theorem~\ref{th:uniqueness} we will need a slightly stronger statement which is implicit in the proof of \cite[Lemma~3.5]{BNRSW16}.

\begin{lem}\label{lem:BNRSW3.5}
	Let $\g$ be a locally compact Hausdorff \'etale groupoid, let $\f\subseteq \iso(\g)^\circ$ be an open subgroupoid containing the unit space, and suppose that $u\in X_\f$. 
	\begin{enumerate}
		\item \label{it1:BNRSW3.5} Let $a\in C^*_r(\g)$. Then for all $\ep>0$ there exists $c\in C^*_r(\f)$ and $b\in C_c(\go)^+$ such that $\|b\|_r = 1$, $\vp(b) = 1$ for all states that factor through $C^*_r(\g_u^u)$ and such that $\|bab - \iota_{\f, r}(c)\|_r<\ep$.
		\item \label{it2:BNRSW3.5} Let $a\in C^*(\g)$. Then for all $\ep>0$ there exists $c\in C^*(\f)$ and $b\in C_c(\go)^+$ such that $\|b\| = 1$, $\vp(b) = 1$ for all states that factor through $C^*(\g_u^u)$ and such that $\|bab - \iota_{\f}(c)\|<\ep$.
	\end{enumerate}  
\end{lem}
\begin{proof}
	As in \cite[Lemma~3.5]{BNRSW16}, the proofs of \ref{it1:BNRSW3.5} and \ref{it2:BNRSW3.5} are exactly the same. 
	
	We first note that the proof of \cite[Lemma~3.3(b)]{BNRSW16} goes through exactly as written there if one replaces $\iso(\g)^\circ$ with $\f$. So given $a\in C^*_r(\g)$ we can find $f\in C_c(\g)$ with $\|f-a\|_r<\ep$ and then use that modification of \cite[Lemma~3.3(b)]{BNRSW16} to find $b\in C_c(\go)^+$ such that $b(u) = \|b\| = 1$ and $bfb\in C_c(\f)$. The quotient map $q_u: C^*_r(\f)\to C^*_r(\g_u^u)$ acts by restricting elements of $C_c(\f)$ to $\g_u^u$, and hence $q_u(b) = 1_{\{u\}}$, the identity of $C^*_r(\g_u^u)$. If $\vp = \psi\circ q_u$ for some state $\psi$ on $C^*_r(\g_u^u)$, then $\vp(b) = \psi(q_u(b)) = \psi(1_{\{u\}}) = 1$. Finally, taking $c = bfb$ gives $\|bab-\iota_{\f, r}(c)\|_r = \|bab- bfb\|_r\leq \|a-f\|_r <\ep$. 
\end{proof}
Note that the main difference in the statement of Lemma~\ref{lem:BNRSW3.5} and \cite[Lemma~3.5]{BNRSW16} is that $b$ can be taken to be in $C_c(\go)^+$ (but again this is implicit in their proof). 
  
\begin{proof}[Proof of Theorem~\ref{th:uniqueness}] The proof of \ref{it1:uniqueness} goes through verbatim from that of \cite[Theorem~3.1(a)]{BNRSW16} using Lemma~\ref{lem:BNRSW3.5}. 
	
	To prove \ref{it2:uniqueness}, for each $u\in X_\f$ let 
	\[
	S_u = \{\psi\circ q_u: \psi \text{ is a pure state on }C^*_r(\g_u^u)\},
	\]
	and let $S := \bigcup_{u\in X_\f}S_u$. Then \ref{it1:uniqueness} implies each element of $S_u$ extends uniquely to a state on $C_r^*(\g)$. For a state $\psi\in S$ with extension $\widetilde\psi$ to $C_r^*(\g)$, let $\pi_{\widetilde\psi}$ denote the GNS representation associated to $\widetilde\psi$. Then by \cite[Theorem~3.2]{BNRSW16} it will be enough to show that $\pi_S := \bigoplus_{\psi\in S}\pi_{\widetilde\psi}$ is faithful. 
	
	Since $\g$ is Hausdorff, there is a faithful conditional expectation $\Phi_\g: C^*_r(\g)\to C_0(\go)$ given on functions by restricting to the unit space. We will be done if we can show that $\pi_S(a) = 0$ implies that $\Phi_\g(a^*a) = 0$. Suppose otherwise, that $\Phi_\g(a^*a) \neq 0$. Since $\Phi_\g(a^*a)$ is a positive function on $\go$ and $X_\f$ is dense in $\go$, we can find $u\in X_\f$ such that $\Phi_\g(a^*a)(u) >0$. Then find positive $\ep$ with $\Phi_\g(a^*a)(u) >\ep>0$ and use Lemma~\ref{lem:BNRSW3.5} to find $b\in C_c(\go)^+$ and $c\in C^*_r(\f)$ such that $\vp\circ q_u(b) = 1$ for all states $\vp$ on $C^*_r(\g_u^u)$ with
	\[
	\|ba^*ab-\iota_{\f, r}(c)\|_r<\ep/2.
	\]
	The argument after \cite[Equation (3.4)]{BNRSW16} goes through verbatim to show that $|\vp(c)|<\ep/2$ for all $\vp\in S_u$. Hence
	\[
	\|q_u(c)\| = \sup\{|\psi(q_u(c))|: \psi\text{ is a pure state on }C_r^*(\g_u^u)\} = \sup\{|\vp(c)|: \vp\in S_u\} \leq \ep/2.
	\]
	
	Because $\f$ is Hausdorff, there is a canonical conditional expectation $\Phi_\f: C^*_r(\f)\to C_0(\go)$ given on functions by restricting to the unit space. For $x\in \go$, let $\ev_x$ denote the character on $C_0(\go)$ given by evaluation at $x$. Then $\ev_x\circ \Phi_\f$ is a state on $C^*_r(\f)$, and the exact same argument as in the proof of \cite[Theorem~3.1]{BNRSW16} shows that for all $x\in X_\f$, $\ev_x\circ \Phi_\f$ factors through $C^*_r(\g_x^x)$. It is straighforward to see that $\Phi_\g\circ\iota_{\f, r} = \Phi_\f$. Hence we have
	\[
	|\ev_u\circ \Phi_\g(\iota_{\f, r}(c))| = |\ev_u\circ\Phi_\f(c)| = |\psi(q_u(c))|\leq \|q_u(c)\| \leq \ep/2.
	\]
	Also, since $\ev_u$ is a character and $\Phi_\g$ is a $C_0(\go)$--bimodule map, we have $\ev_u\circ\Phi_\g(ba^*ab) = \ev_u(b)\ev_u\circ\Phi_\g(a^*a)\ev_u(b) = \ev_u\circ\Phi_\g(a^*a)$. Putting this together gives
	\begin{align*}
		|\ev_u\circ\Phi_\g(a^*a)|&\leq |\ev_u\circ\Phi_\g(a^*a) - \ev_u\circ \Phi_\g(\iota_{\f, r}(c)) + \ev_u\circ \Phi_\g(\iota_{\f, r}(c))|\\
		&\leq |\ev_u\circ\Phi_\g(ba^*ab) - \ev_u\circ \Phi_\g(\iota_{\f, r}(c))| + |\ev_u\circ \Phi_\g(\iota_{\f, r}(c))|\\
		&\leq \|ba^*ab-\iota_{\f, r}(c))\|_r + |\ev_u\circ \Phi_\g(\iota_{\f, r}(c))|\\
		&< \ep.
	\end{align*}
This contradicts our choice of $\ep$, and hence $\Phi_\g(a^*a) = 0$, implying $a= 0$. Hence $\pi_S$ is injective and \cite[Theorem 3.2]{BNRSW16} implies the result.
\end{proof}

\section{Application to inverse semigroups}\label{sec:ISG}

In this section we recall the definition of the tight groupoid of an inverse semigroup, and then apply Theorem~\ref{th:uniqueness} to it. It should be noted that the results of \cite{BNRSW16} have already been comprehensively applied in \cite{LM17} and \cite{LMS19} to give uniqueness theorems for inverse semigroup C*-algebras. As we discuss below (see Lemma~\ref{lem:0dis} and the discussion before and after it) we need a generalization of the results they obtain for one of our applications (the LCM monoids considered in Section~\ref{sec:LCM}). 

We start by giving the relevant definitions, from that of a semigroup all the way up to the tight groupoid of an inverse semigroup. For references for what is stated below, one can see \cite{Pa02}, \cite{Ex08}, and \cite{La98}.  

A {\em semigroup} is a set $S$ equipped with an associative binary operation, and is called a {\em monoid} if it has an identity element. A semigroup is called an {\em inverse semigroup} if for every $s\in S$ there exists a unique $s^*\in S$ such that $ss^*s = s$ and $s^*ss^* = s^*$. The idempotents form a commutative inverse subsemigroup $E(S) = \{e\in S: e^2 = e\}$, with $e^* = e$ for all $e\in E(S)$. A {\em zero element} is an element $0$ such that $0s = s0 = 0$ for all $s\in S$; such an element must be unique if it exists. 

There is a natural partial order on any inverse semigroup $S$ given by $s\leqslant t$ if and only if there exists $e\in S$ such that $se = t$. On $E(S)$ this becomes $e\leqslant f$ if and only if $ef = e$; with this ordering $E(S)$ is a meet-semilattice with $e\wedge f = ef$. If we have $C\subseteq D\subseteq E(S)$, $C$ is called a {\em cover} of $D$ if for any nonzero $e\in D$ we can find $e\in C$ with $ec \neq 0$. If $e\in E(S)$ and $C$ is a cover of $\{f\in E(S): f\leqslant e \}$, we say that $C$ is a cover of $e$. 

If $X$ is a set, the {\em symmetric inverse monoid} on $X$ is $\I(X) = \{f:U\to V: U, V\subseteq X, f\text{ bijective }\}$. It is an inverse semigroup under composition of functions on the largest possible domain, and with $f^* = f^{-1}$. The empty function is the zero element of $\I(X)$, and here $f\leqslant g$ if and only if $g$ extends $f$ as a function. 

A {\em filter} in $E(S)$ is a proper subset $\xi\subset E(S)$ which is {\em upwards closed} in the sense that $e\in \xi$ and $e\leqslant f$ implies $f\in \xi$, and {\em downwards directed} in the sense that $e, f\in \xi$ implies $ef\in \xi$. The set of all filters is denoted $\Ef(S)$; we give it a topology by viewing it as a subset of the power set of $E(S)$ which can be identified with the compact product space $\{0,1\}^{E(S)} $. With this topology, $\Ef(S)$ is called the {\em spectrum} of $E(S)$. 

A filter is called an {\em ultrafilter} if it is not properly contained in any other filter. The subspace of ultrafilters is denoted $\Eu(S)\subseteq \Ef(S)$, its closure is called the space of {\em tight filters} and is denoted $\Et(S) = \overline{\Eu(S)}$. For an element $x\in E(S)$ and finite set $Y\fs E(S)$, let
\[
U(x,Y) := \{\xi\in \Et(S): x\in \xi, Y\cap \xi = \emptyset\}.
\]
Sets of this type form a compact open basis for the topology on $\Et(S)$. The special case where $Y$ is empty will be mentioned frequently in what follows; these we denote
\[
D_e := \{\xi\in \Et(S): e\in \xi\}.
\]
When $\xi$ is an ultrafilter, the sets $\{D_e\}_{e\in \xi}$ form a neighbourhood base for $\xi$. Then the map $\theta: S\to \I(\Et(S))$ defined by 
\[
\theta_s: D_{s^*s}\to D_{ss^*}, \hspace{1cm}\theta_s(\xi) = \{ses^*: e\in \xi\}^{\leqslant}
\]
(where $A^{\leqslant}:= \{e\in E(S): a\leqslant e\text{ for some }a\in A$) is an inverse semigroup homomorphism. Furthermore, each $\theta_s$ is a homeomorphism between compact open subsets of $\Et(S)$. 

We can now define the tight groupoid of $S$. Put an equivalence relation on $S\times \Et(S)$ by declaring $(s, \xi) \sim (t, \xi)$ if and only if there exists $e\in \xi$ such that $se = te$. The set of all equivalence classes is
\[
\gt(S) := \{[s,\xi] : \xi\in D_{s^*s}\}
\]
and is a groupoid with range, source, inverse, and product given by
\[
d[s,\xi] = \xi,\hspace{0.5cm} r[s,\xi] = \theta_s(\xi),\hspace{0.5cm} [s, \xi]^{-1} = [s^*, \theta_s(\xi)],\hspace{0.5cm} [s, \theta_t(\xi)][t,\xi] = [st,\xi]
\]
where we are identifying $\gt(S)^{(0)} = \{[e,\xi] : \xi\in D_e\}$ with $\Et(S)$. 

For a compact open set $U\subseteq D_{s^*s}$ define
\[
\Theta(s, U):= \{[s,\xi]: \xi\in U\}.
\]
These sets generate a topology on $\gt(S)$ for which it is \'etale, and sets of this type form a basis of compact open bisections in $\gt(S)$. 

A summary of what was presented above is that $\theta:S\to \I(\Et(S))$ is an {\em action} of $S$ on $\Et(S)$, and $\gt(S)$ is the {\em groupoid of germs} or {\em transformation groupoid} of this action, see \cite{Pa02, Ex08}.

This groupoid will not always be Hausdorff. By \cite[Theorem~3.16]{EP16}, $\gt(S)$ is Hausdorff if and only if 
\begin{equation}\label{eq:H}\tag{H}
\text{for every $s\in S$, the set $\J_s := \{e\in E(S): e \leqslant s\}$ has a finite cover.} 
\end{equation}
Let $s\in S$ and let $e\leqslant s^*s$. Then following \cite{EP16} we say $e$ is {\em fixed} by $s$ if $se = e$, and we say that $e\leqslant s^*s$ is {\em weakly fixed} by $s$ if $sfs^*f\neq 0$ for all $0<f\leqslant e$. For $s\in S$ we write
\[
\W_s := \{e\in E(S): e\text{ weakly fixed by }s\}.
\]
Note that every nonzero idempotent fixed by $s$ is also weakly fixed by $s$, so $\J_s\setminus \{0\} \subseteq \W_s$. We also note that if $e$ is weakly fixed by $s$, setting $t = se$ gives that $t^*t= es^*se = e$ is weakly fixed by $t$.
\begin{lem}\label{lem:SisoISSG}
Let $S$ be an inverse semigroup. Then the set
\begin{equation}\label{eq:Sisodef}	
S^\iso := \{s\in S: s^*s \text{ is weakly fixed by }s\}
\end{equation}
is a inverse subsemigroup of $S$ which contains $E(S)$.
\end{lem}
\begin{proof}
	First let $s\in S^\iso$ and let $e\leqslant ss^*$ be a nonzero idempotent. Then $s^*es\leqslant s^*s$ is nonzero so by assumption we have $s(s^*es)s^*(s^*es)$ is nonzero. But $e\leqslant ss^*$ implies this is equal to $s^*ese$, so $s^*\in S^\iso$. 
	
	Now we show $S^\iso$ is closed under products. Take $s,t\in S^\iso$ and suppose $e\leqslant (st)^*st$ is a nonzero idempotent. Since $e\leqslant t^*s^*st\leq t^*t$, our assumption implies $tet^*e$ is nonzero. The inequality $e\leqslant t^*s^*st$ also implies $tet^*\leqslant tt^*s^*stt^* = s^*stt^* \leqslant s^*s$, and since $0\neq tet^*e \leqslant tet^*$  our assumption implies $s(tet^*e)s^*(tet^*e)$ is nonzero. This implies $s(tet^*e)s^*e$ is nonzero, and $tet^*e\leqslant tet^*$ then gives $0\neq s(tet^*e)s^*e \leqslant stet^*s^*e$. Thus $stet^*s^*e$ is nonzero, and since $e\leqslant t^*s^*st$ was arbitrary we have $st\in S^\iso$.
\end{proof}
We note that, unknown to us when first posting this preprint, this inverse subsemigroup was previously considered by Li in \cite[Definition~5.8]{Li22}, and there was denoted $S^c$. We keep the notation $S^\iso$ to emphasize its relation to the isotropy group bundle (see below, Proposition~\ref{prop:SisoGiso}) and to avoid confusion with the core of a right LCM semigroup and the inverse semigroup it generates in Section~\ref{sec:LCM}. 

Lemma~\ref{lem:SisoISSG} implies that the set $\{[s,\xi]\in\gt(S): s\in S^\iso, \xi\in D_e\}$ is a subgroupoid of $\gt(S)$ containing the unit space which is easily seen to be isomorphic to $\gt(S^\iso)$. In what follows we will identify this subgroupoid with $\gt(S^\iso)$ without further remark.
\begin{prop}\label{prop:SisoGiso}
	Let $S$ be an inverse semigroup satisfying \eqref{eq:H} and let $S^\iso$ be as in \eqref{eq:Sisodef}. Then \begin{enumerate}
		\item \label{it1:SisoGiso} $\gt(S^\iso)$ is an open subgroupoid of $\iso(\gt(S))^\circ$, 
		\item \label{it2:SisoGiso} the set $\{u\in \Et(S) : \gt(S^\iso)_u^u = \gt(S)_u^u\}$ is dense in $\Et(S)$, and
		\item \label{it3:SisoGiso} $\gt(S^\iso)\subseteq \iso(\gt(S))^\circ\subseteq \overline{\gt(S^\iso)}.$
	\end{enumerate}
\end{prop}
\begin{proof}
	Let $s\in S^\iso$. Then the open set $\Theta(s, D_{s^*s})$ is contained in $\iso(\gt(S))$ by \cite[Lemma~4.9]{EP16}; this shows that $\gt(S^\iso)$ is an open subgroupoid of $\iso(\gt(S))^\circ$.
	
	By the proof of \cite[Lemma~3.3(a)]{BNRSW16}, the set $X:= \{u\in \Et(S) : \gt(S)_u^u = \iso(\gt(S))^\circ_u\}$ has nowhere dense complement. Thus $\overline{X^c}$ is also nowhere dense, so $(\overline{X^c})^c$ is an open dense subset of $\Et(S)$ which is contained in $X$. Then for any open set $U\subseteq \Et(S)$ one can find an ultrafilter $\xi\in U \cap X$. It remains to show that $\iso(\gt(S))_\xi^\circ \subseteq \gt(S^\iso)^\xi_\xi$. Suppose that $[s,\xi] \in \iso(\gt(S))_\xi^\circ$. Then because $\xi$ is an ultrafilter we can find $e\in \xi$ such that $\Theta(s, D_e)\subseteq \iso(\gt(S))^\circ$, whence \cite[Lemma~4.9]{EP16} implies $e$ is weakly fixed by $s$. Hence $[s,\xi] = [se, \xi]\in \gt(S^\iso)$, with both domain and range equal to $\xi$. 
	
	We have left to show that $\iso(\gt(S))^\circ\subseteq \overline{\gt(S^\iso)}
	$. Take $[s,\xi]\in \iso(\gt(S))^\circ$. For every open neighbourhood $U$ of $\xi$ with the property that $\Theta(s, U)\subseteq \iso(\gt(S))$, one can find an ultrafilter $\eta_U\in U$ and hence an idempotent $e_U\in \eta_U$ with $D_{e_U}\subseteq U$. Then again by \cite[Lemma~4.9]{EP16} $e_U$ is weakly fixed by $s$ and so  $[s, \eta_U] = [se_U, \eta_U]$ is an element of $\gt(S^\iso)$. Then the net $([s, \eta_U])_{U}$ indexed by the neighbourhoods of $[s,\xi]$ clearly converges to $[s,\xi]$, implying $[s,\xi]\in\overline{\gt(S^\iso)}$.
\end{proof}
Assuming that \eqref{eq:H} holds, our above descriptions of $C^*_r(\gt(S))$ and $C^*(\gt(S))$ apply. Because the sets $\Theta(s, D_e)$ are compact and open, their characteristic functions are in $C_c(\gt(S))$. We let
\begin{equation}\label{eq:Tdef}
T_s : = 1_{\Theta(s, D_{s^*s})} \in C_c(\gt(S)).
\end{equation} 
The C*-algebra $C^*(\gt(S))$ enjoys a universal property which we describe now. A map $\rho:S\to A$ from an inverse semigroup to a C*-algebra is called a {\em representation} if $\rho(st)= \rho(s)\rho(t)$ and $\rho(s^*) = \rho(s)^*$. Note that if $\rho$ is a representation, then $\rho(E(S))$ is a set of commuting projections in $A$, so given $e, f\in E(S)$ we can form the least upper bound $\rho(e)\vee \rho(f) = \rho(e) + \rho(v) - \rho(e)\rho(v)$. A representation is called {\em tight} if whenever $C$ is a cover of $e$, we have that $\rho(e) = \bigvee_{c\in C}\rho(c)$. This is not the original definition given by Exel \cite[Defintion~13.1]{Ex08}, but was shown to be equivalent in \cite{DM14, Ex19}. Then Exel defined the {\em tight C*-algebra of }$S$, denoted $\Ct(S)$, to be the universal C*-algebra generated by one generator for each element of $S$ subject to the relations that say the map which sends $s\in S$ to its corresponding generator is a tight representation. By \cite[Theorem~2.4]{Ex10} our groupoid C*-algebra $C^*(\gt(S))$ is isomorphic to $\Ct(S)$, and the map $T: S\to C^*(\gt(S))$ given by \eqref{eq:Tdef} is a tight representation with the universal property, i.e., if $\rho: S\to A$ is any other tight representation there exists a $*$-homomorphism $\pi_\rho: C^*(\gt(S))\to A$ with $\rho = \pi\circ T$. 

With all that stated, we now have the following uniqueness theorem for tight C*-algebras of inverse semigroups.
\begin{theo}\label{thm:isguniqueness}
	Let $S$ be an inverse semigroup satisfying \eqref{eq:H}. Let $\{T_s\}_{s\in S}\subseteq C^*_r(\gt(S))$ denote the canonical set of generators. Then if $B$ is a C*-algebra and $\phi: C^*_r(\gt(S))\to B$ is a *-homomorphism, $\phi$ is injective if and only if it is injective on $C^*(T_s: s\in S^\iso)$. In particular, if $\gt(S)$ satisfies weak containment, then a *-homomorphism $\phi: \Ct(S)\to B$ is injective if and only if it is injective on $\Ct(S^\iso)$. 
\end{theo}
\begin{proof}
	Follows directly from Theorem~\ref{th:uniqueness} and  Proposition~\ref{prop:SisoGiso}.
\end{proof}
As mentioned above, \cite[Theorem~3.1(b)]{BNRSW16} has previously been applied to the tight C*-algebra of an inverse semigroup in \cite[Theorem~5.7]{LM17} and \cite[Theorem~7.4]{LMS19}. They have different hypotheses on $S$ to state their uniqueness theorem, so we now clarify the relationship between their results and Theorem~\ref{thm:isguniqueness}. 

An inverse semigroup $S$ is called {\em 0-disjunctive} if whenever $0<e<f$ for idempotents $e$ and $f$, then there exists $0 < e'<f$ with $ee' = 0$. The {\em centralizer} of $S$ is the inverse subsemigroup $Z(S) = \{s\in S: se = es\text{ for all }e\in E(S)\}$. The results \cite[Theorem~5.7]{LM17} and \cite[Theorem~7.4]{LMS19} then state that when $S$ satsfies \eqref{eq:H} and is 0-disjunctive, a $*$-homomorphism $\phi: C^*_{r}(\gt(S))\to B$ is injective if and only if it is injective on $C^*(T_s: s\in Z(S))$. The next lemma shows the relationship between their hypotheses and ours.
\begin{lem}\label{lem:0dis}
	Let $S$ be an inverse semigroup, let $Z(S) = \{s\in S: se = es\text{ for all }e\in E(S)\}$, and let $S^\iso$ be as in \eqref{eq:Sisodef}. Then $Z(S)\subseteq S^\iso$. Furthermore, if $S$ is 0-disjunctive we have the following:
	\begin{enumerate}
		\item \label{it1:0disLemma}$s^*s = ss^*$ for all $s\in S^\iso$.
		\item \label{it2:0disLemma}$ses^* = e$ for all $s\in S^\iso$ with $0<e\leqslant s^*s$.
		\item \label{it3:0disLemma}$Z(S) = S^\iso$.
	\end{enumerate}
\end{lem}
\begin{proof}
	First take $s\in Z(S)$, and suppose $0\neq e\leqslant s^*s$. Then $ses^*e = ses^*$, as $s^*\in Z(S)$ as well. If $ses^* = 0$, then $s^*ses^*s = 0$ as well, but $e\leqslant s^*s$ implies $s^*ses^*s = e$. Since $e$ was assumed to be nonzero, this is a contradiction. Hence $ses^*e \neq 0$ and so $s\in S^\iso$. 
	
	Now assume $S$ is 0-disjunctive. If $s \in S^\iso$, then by definition we have $0 \neq ss^*ss^*(s^*s)  = ss^*s^*s$. If this is equal to $s^*s$ we would have $s^*s\leqslant ss^*$, and a symmetric argument applied to $s^*$ would give the other inequality. Otherwise we have $ss^*s^*s < s^*s$, so by assumption we can find $0<e < s^*s$ with $0 = ss^*s^*se = ss^*e$. Then $s\in S^\iso$ implies $ses^*e \neq 0$. But $ss^*e = 0$  implies $ses^*e = ses^*ss^*e = ses^*(0) = 0$, a contradiction. This implies $s^*s = ss^*$.
	
	To prove \ref{it2:0disLemma} we note that it is clearly true when $e = s^*s$, so we can assume $0<e<s^*s$. Then $ses^*e \neq 0$. Suppose that $ses^* \neq e$. Then we either have $ses^*e < e$ or $ses^*e <ses^*$. In the first case $0<ses^*e<e<s^*s$, so by assumption we can find $0<f<e$ with $0 = ses^*ef = ses^*f$. Since $s\in S^\iso$ we have that $sfs^*f \neq 0$. But $0<sfs^*f \leq ses^*f = 0$, a contradiction. In the second case where $ses^*e <ses^* < ss^*$, so by assumption we can find $0<f<ses^*$ with $0=fses^*e = fe$. Since $s\in S^\iso$ implies $s^*\in S^\iso$, we have that $s^*fsf\neq 0$. But then $0< s^*fsf \leqslant s^*(ses^*)sf = s^*sef = 0$, a contradiction. Contradiction in both cases implies $ses^* = e$.
	
	Now \ref{it3:0disLemma} follows: if $s\in S^\iso$ and $e\leqslant s^*s$, we have $se = ss^*se = ses^*s = es$ by \eqref{it2:0disLemma}. For general $e$, we have $s^*se \leqslant s^*s$ so $se = ss^*se = s^*ses = e(s^*s)s = e(ss^*)s$ by \eqref{it1:0disLemma}, whence $se = es$. 	
\end{proof}
Lemma~\ref{lem:0dis} shows that Theorem~\ref{thm:isguniqueness} generalizes \cite[Theorem~5.7]{LM17} to cases where $S$ might not be 0-disjunctive (though as noted after \cite[Theorem~7.4]{LMS19}, they have ways of getting around the 0-disjunctive assumption). 

The results of \cite{BNRSW16} also give a conditional expectation onto $C^*_{r}(\iso(\gt)^\circ)$ under the additional assumption that $\iso(\gt)^\circ$ is closed. We now give a condition on $S$ which implies this. 

\begin{lem}
	Let $S$ be an inverse semigroup satisfying \eqref{eq:H}. Consider the following statements:
	\begin{enumerate}		
		\item \label{it1:WsClosedLemma} For all $s\in S$ the set $\W_s$ has a finite cover. 
		\item \label{it2:WsClosedLemma} $\iso(\gt(S))^\circ$ is closed.
		\item \label{it3:WsClosedLemma} $\iso(\gt(S))^\circ = \gt(S^\iso)$.
		\item \label{it4:WsClosedLemma} Every tight filter in $E(S)$ is an ultrafilter.
	\end{enumerate} 
Then \ref{it1:WsClosedLemma} $\implies$ \ref{it2:WsClosedLemma}, \ref{it1:WsClosedLemma} $\implies$ \ref{it3:WsClosedLemma},  \ref{it4:WsClosedLemma} $\implies$ \ref{it3:WsClosedLemma}, and \ref{it4:WsClosedLemma} + \ref{it2:WsClosedLemma} $\implies$ \ref{it1:WsClosedLemma}
\end{lem}
\begin{proof}
	Suppose that \ref{it1:WsClosedLemma} holds, take $\gamma = [s,\xi]\in \overline{\iso(\gt(S))^\circ}$, and let $C\subseteq \W_s$ be a finite cover. If some $c\in C$ is in $\xi$, then $[s,\xi] = [sc, \xi]$ with $sc\in S^\iso$, and so $\gamma\in \iso(\gt(S))^\circ$ and we are done. Otherwise we assume that $C\cap \xi = \emptyset$, and consider the open set $U(s^*s, C)$ around $\gamma$. By the definition of closure, $\Theta(s, U(s^*s, C))$ intersects $\iso(\gt(S))^\circ$. Since both are open we can find $e\leqslant s^*s$ with $\Theta(s, D_e)\subseteq \iso(\gt(S))^\circ\cap \Theta(s, U(s^*s, C))$, and so $D_e\subseteq U(s^*s, C) = D_{s^*s} \setminus \bigcup_{c\in C} D_c$ implying $D_e \cap D_c = \emptyset$ for all $c\in C$. But $\Theta(s, D_e)\subseteq \iso(\gt(S))^\circ$ implies $e\in \W_s$, so there must be $c\in C$ with $ce \neq 0$, a contradiction. Thus $\iso(\gt(S))^\circ$ is closed, and \ref{it1:WsClosedLemma} $\implies$ \ref{it2:WsClosedLemma}.
	
	Still supposing \ref{it1:WsClosedLemma} holds, take $[s,\xi]\in \iso(\gt(S))^\circ$. Then we can find a basic open set $U(x, Y)$ around $[s,\xi]$ with $\Theta(s, U(x, Y))$ contained in $\iso(\gt(S))$. If $c\in \xi$ for some $c\in C$ we would have $[s,\xi] = [sc,\xi]$ with $sc\in S^\iso$ and we would be done. Assuming then that $C\cap \xi = \emptyset$, we have that $\xi\in U(x, Y\cup C)\subseteq U(x, Y)$. Since $U(x, Y\cup C)$ is nonempty and open, we can find $e\in E(S)$ such that $D_e\subseteq U(x, Y\cup C)$, implying that $\Theta(s, D_e)\subseteq \iso(\gt(S))$. Then \cite[Lemma~4.9]{EP16} implies $e\in \W_s$. But $D_e\subseteq U(x,Y\cup C)$ implies $ce = 0$ for all $c\in C$, which contradicts the fact that $C$ covers $\W_s$. This contradiction means  $C\cap \xi = \emptyset$ cannot be true, and so as before $sc\in S^\iso$ implies $[s,\xi] = [sc,\xi]\in \gt(S^\iso)$ and \ref{it1:WsClosedLemma} $\implies$ \ref{it3:WsClosedLemma}. 
	
If we suppose that every tight filter is an ultrafilter, then 
$\{D_e: e\in \xi\}$ forms a neighbourhood base for $\xi$ and so for each $[s,\xi]\in \iso(\gt(S))^\circ$ we can find $e\in\xi$ with $\Theta(s, D_e)\subseteq \iso(\gt(S))$ which implies $e\in \W_s$ again by \cite[Lemma~4.9]{EP16}, so as above we have $[s,\xi] = [se,\xi]$ with $se\in S^\iso$. Thus \ref{it4:WsClosedLemma} $\implies$ \ref{it3:WsClosedLemma}. 
	
	Finally, suppose $\iso(\gt(S))^\circ$ is closed and that every tight filter is an ultrafilter, and let $s\in S$. Then $U:=\Theta(s, D_{s^*s})\cap \iso(\gt(S))^\circ$ is a closed subset of the compact set $\Theta(s, D_{s^*s})$, hence is compact. If $[s,\xi]\in U$, then the hypothesis that $\xi$ be an ultrafilter implies there exists $e\in \xi$ with $\Theta(s, D_e)\subseteq U$; again such an $e$ is weakly fixed by $s$. Thus $\{\Theta(s, D_e): e\in \W_s\}$ is an open cover of $U$. Passing to a finite subcover gives $C\fs \W_s$ such that  $\{\Theta(s, D_e): e\in C\}$ covers $U$. If $e\in \W_s$ we have $\Theta(s, D_e)\subseteq U$, which means it intersects $\Theta(s, D_c)$ for some $c\in C$, implying $ec \neq 0$. Hence $C$ is a finite cover for $\W_s$ and \ref{it4:WsClosedLemma} + \ref{it2:WsClosedLemma} $\implies$ \ref{it1:WsClosedLemma}. 
\end{proof}
According to \cite[Proposition~4.1]{BNRSW16}, when $\iso(\g)^\circ$ is closed there is a faithful conditional expectation from $C^*_r(\g)$ to $C^*_r(\iso(\g)^\circ)$ given on $C_c(\g)$ by function restriction. We would like to describe it in our situation at the level of the semigroup. 
\begin{lem}\label{lem:Zdef}
	Suppose $D\subseteq E(S)$ is downwards closed (i.e. $d\in D$ and $e\leqslant d$ implies $e\in D$). Suppose that $B,C\subseteq D$ and that $B,C$ are finite covers of $D$. Then $\cup_{c\in C}D_c = \cup_{b\in B}D_b$. 	
\end{lem}
\begin{proof}
    First suppose that $\xi\in D_c$ for some $c\in C$ and that $\xi$ is an ultrafilter. If $\xi\notin \cup_{b\in B}D_b$, then for all $b$ we can find $e_b\in \xi$  with $e_b\leqslant c$ and $be_b = 0$. Then if we let $e = \prod_{b\in B} e_b\in \xi$, the fact that $D$ is downwards closed implies $e\in D$. But then $be = 0$ for all $b\in B$, contradicting that $B$ is a cover. Hence $\xi \in \cup_{b\in B}D_b$. Now if $\xi\in D_c$ is an arbitrary tight filter, find a net $(\xi_\lambda)$ of ultrafilters in $D_c$ converging to it. By the above argument we have $\xi_\lambda\in\cup_{b\in B}D_b$ for all $\lambda$, and since this set is closed we have that $\xi\in \cup_{b\in B}D_b$. Then $\cup_{c\in C}D_c\subseteq \cup_{b\in B}D_b$, and a symmetric argument gives the other containment.
\end{proof}
So by Lemma~\ref{lem:Zdef}, for any $s\in S$ we can define
\[
Z_s = \bigcup_{c\in C} D_c\hspace{1cm}\text{ where $C$ is any finite cover of $\W_s$}. 
\]
Then $Z_s$ is a compact open subset of $\Et(S)$, and so its characteristic function is an element of $\Ct(S)$. In cases where $\W_s$ is empty we define $Z_s$ to be the empty set, so its characteristic function is 0.

In the following proof, we will use the standard fact that if $U$ and $V$ are compact open bisections in a locally compact \'etale groupoid, then $UV:= \{\gamma\delta: \gamma\in U, \delta\in V\}$ and $U^{-1} := \{\gamma^{-1}: \gamma\in U\}$ are compact open bisections and we have $1_U1_V = 1_{UV}$ and $1_{U^{-1}} = 1_{U}^*$.
\begin{prop}\label{prop:CondExp}
	Let $S$ be an inverse semigroup satisfying \eqref{eq:H}, and suppose that $\W_s$ has a finite cover for all $s\in S$. Then there is a faithful conditional expectation $\mathcal{E}:C^*_r(\gt(S)) \to C^*_r(\gt(S^\iso))$ given on generators by
	\begin{equation}\label{eq:EonSiso}
	\mathcal{E}(T_s) = T_s 1_{Z_s}.
	\end{equation}
\end{prop}
\begin{proof}
	Since $\W_s$ has a finite cover for all $s\in S$, Lemma~\ref{lem:SisoISSG} implies that $\iso(\gt(S))^\circ = \gt(S^\iso)$ is closed. Thus according to \cite[Propositon~4.1]{BNRSW16} there is a faithful conditional expectation onto $C^*_r(\gt(S^\iso))$ given on $C_c(\gt(S))$ by function restriction. Thus, we will be done if we can show that \eqref{eq:EonSiso} does indeed describe function restriction on the set $\{T_s\}_{s\in S}$. 
	
	Recall that $T_s = 1_{\Theta(s, D_{s^*s})}$. So take $[s,\xi]\in \Theta(s, D_{s^*s})$, and first assume $[s,\xi]\in \gt(S^\iso)$. Then there exists $e\in \xi$ with $se\in S^\iso$. If $C$ is any finite cover of $\W_s$, the proof of Lemma~\ref{lem:Zdef} implies that $D_e\subseteq Z_s$, and hence $\xi\in Z_s$. Thus $[s,\xi] = [s, \xi][e,\xi]$ and so $T_s1_{Z_s}([s,\xi]) = T_s1_{Z_s}([s,\xi][e,\xi]) = T_s([s,\xi])$. 
	
	On the other hand, if $[s,\xi]\in \Theta(s, D_{s^*s})\setminus \gt(S^\iso)$, we again have that $\gt(S^\iso) = \iso(\gt(S))^\circ$ is closed and so we can find an open set $U$ around $\xi$ with $[s,\xi] \in \Theta(s, U)$ and $\Theta(s, U)\cap \gt(S^\iso) = \emptyset$. For any ultrafilter $\eta\in U$, supposing that $\eta\in Z_s$ gives that there exists $e\in \eta$ with $D_e\subseteq U\cap Z_s$ (because $Z_s$ is open). If $C$ is a finite cover for $\W_s$, then we would have $ec \neq 0$ for some $c\in C$, and so $ec$ would be weakly fixed by $s$, contradicting $\Theta(s, U)\cap \gt(S^\iso) = \emptyset$. Thus every ultrafilter $\eta\in U$ is not in $Z_s$, and hence $T_s1_{Z_s}([s,\eta]) = T_s1_{Z_s}([s,\eta][e,\eta]) = 0$. Since points of the form $[s,\eta]$ with $\eta$ an ultrafilter are dense in $\Theta(s, U)$, we conclude that $T_s1_{Z_s}$ is zero there, and so also at $[s,\xi]$. 
\end{proof}

\section{Right LCM monoids}\label{sec:LCM}

To every right LCM monoid $P$ one can form an inverse semigroup $S$ generated by the left multiplication functions. In \cite{StLCM} we showed that the boundary quotient C*-algebra $\Q(P)$ is isomorphic to $\Ct(S)$. In this section we apply our results give a uniqueness theorem for $\Q(P)$.

A semigroup $P$ is called {\em left cancellative} if $pq = pr$ implies $q=r$ for all $p,q,r\in P$, and is called {\em right cancellative} if $qp = rp$ implies $q = r$ for all $p,q,r\in P$. A monoid is called {\em right LCM} if it is left cancellative and, for all $p,q\in P$, the intersection of principal right ideals $pP\cap qP$ either empty or equal to $rP$ for some $r\in P$. Recall that for a monoid $P$, its {\em group of units} is $U(P) = \{p\in P: pq = 1 = qp\text{ for some }q\in P\}$. 

The {\em core} submonoid of a right LCM monoid $P$ (see \cite{CL07} and \cite{StLCM}) is 
\[
P_c = \{p\in P: pP\cap qP\neq 0 \text{ for all }q\in P\}.
\]
We proved in \cite[Proposition 4.5]{StLCM} that $P_c$ is a submonoid, that $pq\in P_c$ implies $p,q\in P_c$, and $p,q\in P_c$ and $pP\cap qP = rP$ implies $r\in P_c$.

We now give the definition of its universal C*-algebra of $P$ in the special case that $P$ is an LCM monoid. This definition is due to Li \cite{Li12}. Let $\J(P) := \{pP:p\in P\}\cup\{\emptyset\}$ and call this the set of {\em constructible ideals}.  Then the {\em universal C*-algebra of $P$}, denoted $C^*(P)$ is the universal C*-algebra generated by a set of isometries $\{s_p\}_{p\in P}$ and projections $\{e_X\}_{X\in \J(P)}$ subject to the relations
\begin{enumerate}
	\item $s_ps_q = s_{pq}$ for all $p,q\in P$, 
	\item $v_pe_Xv_{p}^* = e_{pX}$ for all $p\in P$ and $X\in \J(P)$, 
	\item $e_P = 1$, $e_\emptyset = 0$, and
	\item $e_Xe_Y = e_{X\cap Y}$ for all $X, Y\in \J(P)$. 
\end{enumerate}
A {\em foundation set} is a finite subset $F\subseteq P$ such that for all $p\in P$ there exists $f\in P$ such that $fP\cap pP\neq \emptyset$. Then the {\em boundary quotient of $C^*(P)$}, denoted $\Q(P)$, is the universal C*-algebra generated by sets $\{s_p\}_{p\in P}$ and $\{e_X\}_{X\in \J(P)}$ with the same relations as above, and also satisfies
\[
\bigvee_{f\in F}e_{fP} = 1\hspace{1cm}\text{for all foundation sets }F\subseteq P.
\]
We proved in \cite{StLCM} that $\Q(P)\cong \Ct(S)$ for an inverse semigroup $S$ derived from $P$. We will use this isomorphism to prove the following.
\begin{theo}\label{th:LCMmain}
	Let $P$ be an LCM monoid, let $\Q(P)$ denote its boundary quotient C*-algebra with canonical generators $\{s_p: p\in P\}$ and let $\Q_c(P) = C^*(s_p: p\in P_c)$ be the subalgebra generated by the core submonoid. Suppose that $\Q(P) \cong C^*_r(\gt(S))$ and that $S$ satisfies \eqref{eq:H}. Then a $*$-homomorphism $\pi: \Q(P)\to B$ is injective if and only if it is injective on $\Q_c(P)$. 	
\end{theo}

The proof is not immediate from Theorem~\ref{thm:isguniqueness}; the rest of this section will be taken to prove it. 

First, a word about the Hausdorff condition is in order. In \cite[Proposition~4.1]{StLCM} we translated the condition \eqref{eq:H} in this paper to one on $P$. As mentioned in \cite{StLCM}, this will certainly be satisfied when $P$ is cancellative (when embedded in a group, say) but also holds when the counterexamples to right cancellativity have a finite cover, in some sense. See \cite[Equation~(H)]{StLCM} for the precise statement

Fix, for the rest of this section, an LCM monoid $P$. We first recall the form of the inverse semigroup $S$ given in \cite[Proposition 3.2]{StLCM}. Put an equivalence relation $\sim$ on $P\times P$ by saying $(p,q)\sim (a,b)$ if and only if there exists $u\in U(P)$ with $pu = a$ and $qu = b$. Let $S = \{[p,q]: p,q\in P\}\cup\{0\}$ where 0 is an ad-hoc zero element. Then $S$ becomes an inverse monoid when given the operation
\[
[p,q][r,t] = \begin{cases}[pq', tr'] & qP\cap rP = \ell P\text{ with }qq'=rr' = \ell\\
0&qP\cap rP = \emptyset\end{cases}.
\]
Then we have $[p,q]^* = [q,p]$ and $E(S) = \{[p,p]: p\in P\}\cup\{0\}$. We note that the LCM property implies that $\J(P) = \{pP: p\in P\}\cup\{\emptyset\}$ is a semilattice under intersection and is isomorphic to $E(S)$. In what follows we  use this identification, and so write
\[
D_{pP} = \{\xi\subset \J(P): \xi\text{ is a tight filter and } pP\in \xi\}.
\]
For any $X\in \J(P)$ and $p\in P$, write $p^{-1}X:= \{y\in P: py\in X\}$. The LCM property implies that $p^{-1}X\in \J(P)$. Then we have
\[
\theta_{[p,q]}(\xi) = \{p(q^{-1}X): X\in \xi)\}^{\leqslant}.
\]
Referring to \cite[Lemma~4.9]{StLCM}, we have that 
\[
S^\iso =\{[p,q]: paP\cap qaP \neq \emptyset\text{ for all }a\in P\}. 
\]
We would like to relate this to another inverse subsemigroup defined in \cite{StLCM}
\[
S_c = \{[p,q]: p,q\in P_c\}.
\]
In general, these will not be equal. However, we can relate them with the following two lemmas.  
\begin{lem}\label{lem:LCM2}
	Suppose that $[p,q]\in S^\iso$ with $pP\cap qP = rP$. Then $D_{rP} = D_{pP} = D_{qP}$. Furthermore, if either $bP\cap qP \neq \emptyset$ or $bP\cap pP\neq \emptyset$, then $bP\cap rP\neq\emptyset$. 
\end{lem}
\begin{proof}
	First we note $r$ always exists, taking $a = 1$ in the definition of $S^\iso$. Since $rP\subseteq qP, pP$, we have $D_{rP}\subseteq D_{qP}, D_{pP}$. If $\xi\in D_{qP}$, then $\xi$ is a fixed point for $\theta_{[p,q]}$, and hence $\xi = \theta_{[p,q]}(\xi)\in D_{pP}$, so $pP\in \xi$ as well. Since $\xi$ is closed under intersections, we have $rP\in \xi$ so $D_{qP}\subseteq D_{rP}$. A similar argument gives $D_{pP}\subseteq D_{rP}$, and so they are all equal. Now suppose that $bP\cap qP \neq \emptyset$ and find $\xi\in D_{bP}\cap D_{qP}$. Then since $\theta_{[p,q]}$ fixes $\xi$ we must have $pP\in\xi$ as well, which implies $pP\cap qP = rP\in \xi$ and so $rP\cap bP\neq \emptyset$. 
\end{proof}
The proof of the next lemma follows the same lines as that of \cite[Proposition~4.7]{StLCM}, which in turn follows along the same lines as that of \cite[Proposition~5.5]{CL07}.
\begin{lem}\label{lem:LCM1}
	Suppose that $[p,q]\in S^\iso$ with $pP\cap qP = rP$, and let $b \in rP$. Then $[1,b][p,q][b,1]\in S_c$. 
\end{lem}
\begin{proof}
	 We note that $[b,b]\leqslant[r,r]\leqslant [q,q]$, so every point in $D_{bP}$ is fixed by $\theta_{[p,q]}$ and hence also by $\theta_{[q,p]}$.
	
	Write $pp_1 = qq_1 = r$ and $b = rr_1$ for some $p_1, q_1, r_1\in P$. If $\xi\in D_{bP}$ then $\theta_{[p,q]}$ fixes $\xi$ and we have that $p(q^{-1}(bP)) =p(q^{-1}(qq_1r_1P)) = pq_1r_1P\in \xi$, and so also intersects $bP$. Then $bP\cap pq_1r_1P = \ell P$ implies there exist $b_1, p_3\in P$ with $bb_1 = pq_1r_1p_3 = \ell$. A straightforward calculation from this gives
	\[
	[1,b][p,q][b,1] = [b_1, p_3].
	\] 
	Our goal is to show that $b_1, p_3\in P_c$. Going for a contradiction, suppose otherwise. If $b_1\notin P_c$, then find $z\in P$ such that $b_1P\cap zP = \emptyset$. Let $\xi\in D_{bP}$ be any ultrafilter, which we recall must be fixed by $\theta_{[p,q]}$. Define $\xi_b = \theta_{[bz,1]}(\xi)$, an ultrafilter (by \cite[Proposition~3.5]{EP16}) which contains $bzP$ and hence also $bP$, which means it must also be fixed by $\theta_{[p,q]}$. Thus $p(q^{-1}(bP)) = pq_1r_1P\in \xi_b$ as above, which implies $\ell P\in \xi_p$. But $\ell P\cap bzP = bb_1P\cap bzP = \emptyset$ by assumption, which contradicts $\ell P, bzP\in \xi_b$. 
	
	On the other hand, if $p_3\notin P_c$, again find $z\in P$ with $p_3P\cap zP = \emptyset$, let $\xi\in D_{bP}$ be any ultrafilter and let $\xi_b = \theta_{[bz,1]}(\xi)$. Then as above $\xi_b$ contains $bzP$ and $\ell P = pq_1r_1p_3P$. Since $\xi_b$ is fixed by $\theta_{[q,p]}$ it must contain $q(p^{-1}(pq_1r_1p_3P))= qq_1r_1p_3P = bp_3P$, which again leads to a contradiction. Hence $[b_1, p_3]\in S_c$ and we are done.
\end{proof}
Before getting to the proof of Theorem~\ref{th:LCMmain}, we recall that in \cite{StLCM}, we proved that the map $T_{[p,q]}\mapsto s_ps_q^*$ extends to a $*$-isomorphism of $\Ct(S)$ onto $\Q(P)$. Then $\Q_c(P)$ can be identified with the subalgebra generated by $\{T_{[p,q]}\}_{[p,q]\in S_c}$. We use this identification in the proof.
\begin{proof}[Proof of Theorem~\ref{th:LCMmain}]
	Suppose that $\pi: \Q(P)\to B$ is a $*$-homomorphism, and that $\pi$ is injective on $\Q_c(P)$. We wish to show that $\pi$ is injective, and by Theorem~\ref{thm:isguniqueness} it will be enough to show that it is injective on $C^*(T_{[p,q]}: [p,q]\in S^\iso)$. Let $F\fs S^\iso$, and for an arbitrary element $f\in F$ write $f = [p_f, q_f]$ for $p_f, q_f\in P$. Consider a finite linear combination
	\[
	a = \sum_{f\in F}\lambda_f T_{[p_f,q_f]} \hspace{1cm} \lambda_f\in \CC.
	\]
	We suppose that $\pi(a) = 0$ and show that this must imply $a = 0$. Denseness of such elements will then give the result. 
	
	As a function on $\gt(S)$, this is a linear combination of characteristic functions on compact open bisections: $T_{[p_f,q_f]} = 1_{\Theta([p,q], D_{r_{f}P})}$, and so the support of $a$ can be written as disjoint pieces of the following form on which it is constant:
	\begin{equation}\label{eq:disjointification}
	\bigcap_{f\in F'}\Theta([p_f, q_f], D_{r_fP}) \setminus \left(\bigcup_{g\in (F')^c}\Theta([p_g, q_g], D_{r_fP})\right)
	\end{equation}
	We will show $a$ is 0 on each such piece. It will be enough to show it is 0 on each $\gamma\in \gt(S)$ such that $r(\gamma) = d(\gamma)= \xi$ is an ultrafilter, as each set \eqref{eq:disjointification} is open and such points are dense. For such an $F'$, $\gamma$ and $\xi$ we must have $r_fP\in \xi$ for all $f\in F'$ and $r_gP\notin\xi$ for all $g\in (F')^c$. Because $\xi$ is an ultrafilter, for each $g\in (F')^c$ we can find $k_g\in P$ such that $k_gP\cap r_gP = \emptyset$ and $k_gP\in \xi$. Then we have 
	$$\bigcap_{f\in F'}r_fP\cap \bigcap_{g\in (F')^c}k_gP\neq \emptyset,$$
	and this intersection must be of the form $bP$ for some $b$; we must have $bP\subseteq r_fP$ for all $f\in F'$ and $bP\cap r_gP = \emptyset$ for all $g\in (F')^c$. Note that $bP\in \xi$.
	
	Then Lemma~\ref{lem:LCM2} implies $T_{[b,1]}^*T_{[p_g,q_g]}T_{[b,1]} = 0$ for all $g\notin F'$, while Lemma~\ref{lem:LCM1} implies that $T_{[b,1]}^*T_{[p_f,q_f]}T_{[b,1]} \in \Q_c(P)$ for all $f\in F'$. Thus $T_{[b,1]}^*aT_{[b,1]}\in \Q_c(P)$, and $\pi(a) = 0$ implies $\pi\left(T_{[b,1]}^*aT_{[b,1]}\right) = 0$. Our hypothesis now gives that $T_{[b,1]}^*aT_{[b,1]} = 0$. 
	
	Since $\Theta([b,1], D_{bP})$ is a bisection and $\xi$ is in its range, there is a unique $\gamma_b\in \Theta([b,1], D_{bP})$ with range $\xi$. Now the function $T_{[b,1]}^*aT_{[b,1]} = 1_{\Theta([b,1], D_{bP})^{-1}}a1_{\Theta([b,1], D_{bP})}$ applied to $\gamma_b^{-1}\gamma\gamma_b$ is
	\[
	0 = T_{[b,1]}^*aT_{[b,1]}(\gamma_b^{-1}\gamma\gamma_b) = a(\gamma)
	\]
	Since the piece \eqref{eq:disjointification} we chose was arbitrary, we must have $a = 0$ and we are done.
\end{proof}
\begin{rmk} We compare Theorem~\ref{th:LCMmain} with previous uniqueness theorems for semigroup C*-algebras: \cite[Theorem~7.4]{BLS16} in the LCM case and \cite[Theorem~5.1]{LS22} in the group-embeddable case. In the nuclear case, $\Q(P)$ is a quotient of the C*-algebras considered for their uniqueness theorems, so it is natural that ours requires fewer conditions. We also note that, in contrast to the above results, our subalgebra $\Q_c(P)$ does not contain the subalgebra $C^*(e_X: X\in \J(P))$ except in the case where every element of $P$ is invertible.
\end{rmk} 

\begin{rmk}
	If $p\in P_c$, then $\{p\}$ is a foundation set and so the defining relations for $\Q(P)$ imply that $s_p$ is a unitary. So $\Q(P)$ is generated by the group of unitaries corresponding to these elements. In the case $P_c = \{1\}$ we get that $\Q(P)$ is simple, recovering \cite[Theorem~4.12]{StLCM}. 
\end{rmk}
\section{Subshifts}\label{sec:subshifts}

In this section we apply our uniqueness theorem to the tight C*-algebra of an inverse semigroup $\s_\X$ we associated to a subshift $\X$ in \cite{ShiftISG}. There we erroneously claimed that $\Ct(\s_\X)$ is isomorphic to the Carlsen-Matsumoto algebra $\OX$. Here we use our uniqueness theorem to show that it is in general isomorphic to a certain quotient generated by translation operators on $\ell^2(\X)$. 

Subshifts are a class of tractable dynamical systems. If $\A$ is a finite set (called the {\em alphabet}) given the discrete topology, the product space $\A^\NN$ is a compact metrizable totally disconnected space. The shift map $\sigma: \A^\NN\to \A^\NN$ given by
\[
\sigma(x)_i = x_{i+1}
\]
is a continuous map. Any closed subspace $\X\subseteq \A^\NN$ satisfying $\sigma(\X)\subseteq \X$ is called a {\em one-sided subshift} or simply a {\em subshift}. For $n\geq 1$, the set $\A^n$ is called the set of {\em words of length $n$}. We call $\epsilon$ the {\em empty word} and set $\A^0:= \{\epsilon\}$. If $\alpha\in \A^n$, we write $|\alpha| = n$ and call this the length of $\alpha$. The set of all words will be written as $\A^* = \cup_{i=0}^\infty \A^i$. For $\alpha \in \A^*$, define $C(\alpha) = \{x\in \X : x_i = \alpha_i, i = 1, \dots, |\alpha|\}$. Sets of this form are called {\em cylinder sets}; they are clopen subsets of $\X$ and form a basis for its topology. For $\alpha,\beta\in \A^*$, let
\begin{align*}
C(\alpha, \beta)&:= C(\beta)\cap\sigma^{-|\beta|}\left(\sigma^{|\alpha|}(C(\alpha))\right)\\
& = \{\beta x \in \X: \alpha x\in \X\}.
\end{align*}
Since $\sigma$ is a closed map by the closed map lemma, $C(\alpha, \beta)$ is closed. For a finite set $F\fs \A^*$ and $\alpha\in \A^*$, define
\[
C(F; \alpha) = \bigcap_{f\in F} C(f, \alpha).
\]
Each such set is closed as well. For each $a\in \A$, define $s_a: C(a, \epsilon)\to C(a)$ by $s_a(x) = ax$. Then each $s_a$ is a bijection between subsets of $\X$, and so an element of $\I(\X)$. Likewise, for each nonempty word $\alpha\in\A^*$ we define $s_\alpha: C(\alpha, \epsilon) \to C(\alpha)$ by $s_\alpha(x) = \alpha x$, and with the product in $\I(\X)$ we have $s_\alpha = s_{\alpha_1}s_{\alpha_2}\cdots s_{\alpha_{|\alpha|}}$. In \cite{ShiftISG} we defined $\s_\X$ to be the inverse semigroup generated by the $\{s_\alpha\}_{\alpha\in\A^*}$ (with $s_\epsilon$ taken to be the identity element). We proved that 
\begin{equation}\label{eq:SXdef}
\s_\X = \{s_\alpha E(F;\gamma) s_\beta^*: \alpha, \beta, \gamma\in \A^*, F\fs\A^* \}\cup\{0\}
\end{equation}
where $E(F;\alpha) := \id_{C(F;\alpha)}$, and that $E(\s_\X) = \{E(F;\gamma) : \gamma\in \A^*, F\fs\A^* \}\cup\{0\}$. There it is also noted that 
\[
E(F, \alpha) = s_\alpha \left(\prod_{f\in F}s_f^*s_f\right)s_\alpha^*
\]
which implies that if we have $s = s_\alpha E(F;\gamma) s_\beta^*$ and $\alpha = \beta$, then $s$ is an idempotent. We will use this fact in what follows.

Unfortunately, many of the results in \cite{ShiftISG} are incorrect as stated. The problems all stem from \cite[Remark~3.4]{ShiftISG}, which says that the form $s_\alpha E(F;\gamma) s_\beta^*$ is essentially unique. Specifically it says (erroneously) that if $s_\alpha E(F;\gamma) s_\beta^* = s_\delta E(G;\tau) s_\sigma^*$ and the elements are written in {\em standard form} (meaning $\alpha_{|\alpha|}\neq \beta_{|\beta|}$ and $\delta_{|\delta|}\neq \tau_{|\tau|}$), then $\alpha = \delta$ and $\beta = \tau$. This statement is seen to be false if one considers the full shift on one letter $\A = \{a\}$; the functions $s_a$, $s_{aa}^*$ and $E(\{a\}, a)$ are all in standard form. We show in a correction to \cite{ShiftISG}\footnote{Submitted to J. Algebra and appended to the arXiv entry \href{https://arxiv.org/abs/1505.01766}{1505.01766}} that all the results in \cite{ShiftISG} become true if one assumes that $\X$ satisfies {\em condition (I)} of Matsumoto \cite[Section~5]{Mats99}. 

One of the incorrect results is \cite[Theorem~4.8]{ShiftISG}, which states that a certain C*-algebra associated to $\X$, denoted $\OX$ and called the {\em Carlsen-Matsumoto algebra of $\X$} \cite{CS07, Ca08}, is isomorphic to $\Ct(\s_\X)$. Having now established that this isomorphism is not true in general, we turn to the problem of identifying what $\Ct(\s_\X)$ is. Specifically, is there any other C*-algebra associated to a subshift that is isomorphic to $\Ct(\s_\X)$ for subshifts $\X$ which might not satisfy (I)?

The answer comes from another work of Matsumoto. In \cite[Lemma~4.1]{Mats00}, he considers $\ell^2(\X)$, the Hilbert space spanned by an orthonormal basis $\{\delta_x\}_{x\in \X}$ indexed by $\X$. Then for each $a\in \A$, $S_a$ is defined to be the operator 
\begin{equation}\label{eq:SaOperatorDef}
S_a\delta_x= \begin{cases}
\delta_{ax}&\text{if }ax\in \X\\
0&\text{otherwise}\end{cases}.
\end{equation}
The C*-algebra generated by these operators is not given a name in \cite[Lemma~4.1]{Mats00}, so we will denote it $C^*(S_a:a\in \A)$. In what remains of this section, we will use our uniqueness Theorem~\ref{thm:isguniqueness} to prove that $\Ct(\s_\X)\cong C^*(S_a:a\in \A)$.

In what follows, for an element $s\in \I(\X)$, we write $\dom(s)$ for the domain of $s$ and $\ran(s)$ for the range of $s$. 
\begin{lem}\label{lem:shiftfixonepoint}
	Let $\X$ be a subshift. If $s = s_\alpha E(F;v) s_\beta^*$ and $|\alpha| \neq |\beta|$, then $s$ fixes at most one point. Consequently, $\s_\X$ satisfies \eqref{eq:H}.
\end{lem}
\begin{proof}
	Write $s = s_\alpha E(F;v)s_\beta^*$ for $F\fs \A^*$ and $\alpha, \beta, v\in \A^*$, and suppose that $|\alpha|>|\beta|$. Take an arbitrary $\beta x$ in the domain of $s$ and suppose  $s_\alpha E(F;v) s_\beta^*(\beta x) = \beta x$. Then $\alpha x = \beta x$, and so $\alpha = \beta\gamma$ for some $\gamma\in \A^*$. Thus $\beta x = \beta \gamma x$, implying $x = \gamma x$. This implies $x = \gamma^\infty$, and so our arbitrary point in the domain of $s$ must be $\beta\gamma^\infty$. The case $|\beta |>|\alpha|$ follows from the fact that a point is fixed by $s$ if and only if it is fixed by $s^*$. 
	
	Now suppose we have  $e\in E(\s_\X)$ with $se = e$. Write $e = E(G; w)$ for $G\fs \A^*$ and $w\in \A^*$, and let $s$ be as above. If $|\alpha| = |\beta|$, then this implies $\alpha = \beta$ and $s$ is an idempotent, implying that $\{s\}$ is a cover for $\J_s$. Otherwise, let $wx\in E(G; w)$ be an arbitrary element. Then the equality of functions $se = e$ implies that $s(wx) = wx$, and so the above implies that $C(G; w)= \{wx\}$. Since $e$ was an arbitrary element of $\J_s$, we see that $\J_s$ has only one nonzero element, and so of course has a finite cover.	 
\end{proof}

In \cite[Lemma~4.4]{ShiftISG} (which we note is correct as stated), we proved that every ultrafilter in $E(\s_\X)$ is of the form $\xi_x = \{E(F;v): x\in C(F;v)\}$ for some $x\in \X$, and that each $x\in \X$ determines such an ultrafilter.
\begin{lem}\label{lem:ShiftFixedUltra}
	Let $x\in \X$ and let $\xi_x$ be the corresponding ultrafilter. Then $\theta_s(\xi_x) = \xi_{s(x)}$ for all $s\in \s_\X$ with $x\in \dom(s)$. 
\end{lem}
\begin{proof}
	Let $e\in \xi_x$, so that $x\in \dom(e)$. Since $x\in \dom(s)$, we can assume $e\leqslant s^*s$ (which implies $\dom(e) = \dom(se)$). Then $s(x)\in \dom(es^*)$, and since everything in $\ran(e)$ is in $\dom(s)$ we have $s(x) \in \dom(ses^*)$. This implies $ses^*\in \xi_{s(x)}$. Now if $f\geqslant ses^*$ for some $e\in \xi_x$, then $ses^* = s(es^*s)s^*$ with $es^*s\leqslant s^*s$, then by the above $s(x)\in \dom(ses^*)\subseteq \dom(f)$. This gives $\theta_s(\xi_x) \subseteq \xi_{s(x)}$. Since $\theta_s(\xi_x)$ is an ultrafilter this containment must be equality. 	
\end{proof}

\begin{lem}\label{lem:shiftisoIdem}
	Let $\X$ be a subshift. Then $\s_\X^\iso = E(\s_\X)$. 	
\end{lem}
\begin{proof}
	We always have $E(\s_\X)\subseteq \s_\X^\iso$, so suppose $s\in \s_\X^\iso$. Write $s = s_\alpha E(F;v)s_\beta^*$ for $F\fs \A^*$ and $\alpha, \beta, v\in \A^*$. Then $\theta_s$ fixes every tight filter in $D_{s^*s}$ by \cite[Lemma~4.9]{EP16}, and in particular fixes every ultrafilter. By Lemma~\ref{lem:ShiftFixedUltra}, $s$ must fix every $x\in \dom(s^*s)$. If $|\alpha| = |\beta|$ this implies $\alpha = \beta$ and so $s\in E(\s_\X)$. Otherwise, Lemma~\ref{lem:shiftfixonepoint} implies $\dom(s^*s)$ is only one point. But $\dom(s) = \dom(s^*s)$, which consists of one point that $s$ fixes. Thus $s$ is an idempotent function, and we are done.
\end{proof}
\begin{lem}\label{lem:covertojoin}
	Let $\X$ be a subshift, and suppose $C\subseteq E(\s_\X)$ is a cover of $e\in E(\s_\X)$. Then $\bigcup_{c\in C}\dom(c)= \dom(e)$.
\end{lem}
\begin{proof}
	Since $C$ is a cover of $e$, we have $c\leqslant e$ for all $c\in C$ and so $\dom(c)\subseteq \dom(e)$, implying $\bigcup_{c\in C}\dom(c)\subseteq \dom(e)$. So take $x\in \dom(e)$. For each $n\in \NN$, find $c_n\in C$ such that $c_nE(\emptyset,x_1\cdots x_n) \neq 0$. This implies that for each $n$, we can find $y_n\in \dom(c_n)$ such that $y_n$ agrees with $x$ up to position $n$. The sequence $(c_n)_{n\in\NN}$ has finite range, so has a constant subsequence $c_{n_k} = c_0$ for some $c_0\in C$. Then the sequence $y_{n_k}$ converges to $x$ and is contained in $\dom(c_0)$, a closed set. Thus $x\in \dom(c_0)$ and we are done.
\end{proof}

\begin{theo}\label{th:subshiftisg}
	Let $\X$ be a subshift, let $\s_\X$ be the inverse semigroup defined in \eqref{eq:SXdef}, and let $S_a$ be the operator given in \eqref{eq:SaOperatorDef}. Then $\Ct(\s_\X)$ is isomorphic to $C^*(S_a: a\in \A)$. 
\end{theo}
\begin{proof}
	By \cite[Theorem~7.2]{Ca08} and \cite[Proposition~4.6]{ShiftISG} (which we note is correct as stated), there is a surjective $*$-homomorphism $\OX\to \Ct(\s_\X)$. By \cite[Theorem~21]{CS07} $\OX$ is nuclear, and since quotients of nuclear C*-algebras are nuclear we conclude $\Ct(\s_\X)$ is nuclear, and for the same reason we conclude $C^*_r(\gt(\s_\X))$ is nuclear. Hence by \cite[Theorem~5.6.18]{BO08} (for example) $\gt(\s_\X)$ is amenable, implying that $C^*_r(\gt(\s_\X)) \cong \Ct(\s_\X)$. 
	
	It is clear that the map $\rho: \s_\X\to \B(\ell^2(\X))$ given by 
		\[
		\rho(s)\delta_x = \begin{cases} \delta_{s(x)}&x\in\dom(s)\\
		0&x\notin\dom(s)\end{cases}
	\]
	is a representation of $\s_\X$. To see that $\rho$ is tight, by \cite{DM14} we only need to check that whenever $C$ is a cover of $e\in E(\s_\X)$ we have that $\bigvee_{c\in C} \rho(c) = \rho(e)$. But this is immediate from Lemma~\ref{lem:covertojoin}. Thus $\rho$ induces a $*$-homomorphism $\pi: \Ct(\s_\X)\to \B(\ell^2(\X))$. We wish to show that $\pi$ is injective, and by  Theorem~\ref{thm:isguniqueness} and Lemma~\ref{lem:shiftisoIdem} it is enough to show that it is injective on the commutative subalgebra generated by the idempotents. 
	
	Suppose that $\pi: C(\Et(\s_\X))\to \B(\ell^2(\X))$ has nonzero kernel, and let $Y\subseteq \Et(\s_\X)$ be the corresponding closed subset such that $f\in \ker(\pi)$ if and only if $\left.f\right|_Y = 0$. Since $Y^c$ is open, we can find $D_e\subseteq Y$ for some $e\in E(\s_\X)$. Since $1_{D_e}\in C(\Et(\s_\X))$ and is 0 on $Y$, $\pi$ must send it to 0, which is absurd. Hence the restriction of $\pi$ to $C(\Et(\s_\X)) = \Ct(\s_\X^\iso)$ is injective, from which Theorem~\ref{thm:isguniqueness} implies $\pi$ is injective. The last thing to see is that the range of $\pi$ is $C^*(S_a: a\in \A)$, but this is immediate upon seeing that $\sum_{a\in \A}S_aS_a^*$ is the identity operator.
\end{proof}
We note that this aligns with \cite[Remark~7.5]{Ca08}.
\begin{rmk}
	Lemma~\ref{lem:shiftisoIdem} implies that $\gt(\s_\X)$ is the unit space $\Et(\s_X)$, and since Lemma~\ref{lem:shiftfixonepoint} implies $\gt(\s_\X)$ is Hausdorff, the unit space is closed. Hence Proposition~\ref{prop:SisoGiso}\ref{it3:SisoGiso} implies that $\iso(\gt(\s_\X))^\circ = \Et(\s_X) = \gt(\s_\X)$, so that $\gt(\s_\X)$ is effective. Hence, once effectiveness was established, one could have instead appealed to e.g. \cite[Proposition~5.5]{BCFS14}.
\end{rmk}

{\bf Acknowledgement:} We thank Toke Meier Carlsen for informing us that the results of \cite{ShiftISG} are not correct as stated and for conjecturing that Theorem~\ref{th:subshiftisg} might be the correct formulation. We also thank Chris Bruce and Sven Raum for comments and suggestions on a preliminary version of this work.
\bibliographystyle{acm}
\bibliography{E:/Dropbox/Research/bibtex}{}

{\small 
\textsc{Carleton University, School of Mathematics and Statistics. 4302 Herzberg Laboratories} \texttt{cstar@math.carleton.ca} 
}
\end{document}